\documentclass[a4paper,12pt,oneside]{amsart}

\usepackage[applemac]{inputenc}
\usepackage[british]{babel}
\usepackage{fancyhdr}
\usepackage[height=21cm]{geometry}
\usepackage{setspace}
\usepackage{bbm}
\usepackage{amssymb}
\usepackage[all]{xy}

\usepackage{color}
\definecolor{citation}{rgb}{0.2,0.3,1}
\usepackage[final,colorlinks,citecolor=citation,pagebackref]{hyperref}

\newlength{\aufzleft}
\newenvironment{aufz}{\begin{list}{}{\setlength{\listparindent}{0pt}\setlength{\itemsep}{\topsep}\setlength{\labelwidth}{3.2ex}\setlength{\aufzleft}{\labelsep}\addtolength{\aufzleft}{\labelwidth}\setlength{\leftmargin}{\aufzleft}}}{\end{list}}
\newenvironment{equi}{\begin{list}{}{\setlength{\listparindent}{0pt}\setlength{\itemsep}{\topsep}\setlength{\labelwidth}{4.1ex}\setlength{\aufzleft}{\labelsep}\addtolength{\aufzleft}{\labelwidth}\setlength{\leftmargin}{\aufzleft}}}{\end{list}}

\newtheoremstyle{par}{1ex}{2ex}{\rm}{}{\bfseries}{}{0.8em}{\thmnumber{(#2)}}
\newtheoremstyle{thm}{1ex}{2ex}{\itshape}{}{\bfseries}{}{0.9em}{\thmnumber{(#2)}\thmname{ #1}\thmnote{ (#3)}}
\newtheoremstyle{ex}{1ex}{2ex}{\rm}{}{\bfseries}{}{0.8em}{\thmnumber{(#2)}\thmname{ #1}}

\theoremstyle{par}
\newtheorem{no}{}[section]

\theoremstyle{thm}
\newtheorem{lemma}[no]{Lemma}
\newtheorem{prop}[no]{Proposition}
\newtheorem{cor}[no]{Corollary}
\newtheorem{thm}[no]{Theorem}

\theoremstyle{ex}
\newtheorem{qu}[no]{Question}
\newtheorem{qus}[no]{Questions}

\newcommand{\N}{\mathbbm{N}}
\newcommand{\Z}{\mathbbm{Z}}
\newcommand{\Q}{\mathbbm{Q}}
\newcommand{\dfgl}{\mathrel{\mathop:}=}
\newcommand{\loccit}{\textit{loc.\,cit.\ }}
\newcommand{\hm}[3]{{\rm Hom}_{#1}(#2,#3)}
\newcommand{\ga}{\Gamma}
\newcommand{\oga}{\overline{\Gamma}}
\newcommand{\Id}{{\rm Id}}
\newcommand{\catmod}{{\sf Mod}}
\newcommand{\res}{\!\upharpoonright}
\newcommand{\ia}{\mathfrak{a}}
\newcommand{\ib}{\mathfrak{b}}
\newcommand{\ic}{\mathfrak{c}}
\newcommand{\ip}{\mathfrak{p}}
\newcommand{\iq}{\mathfrak{q}}
\newcommand{\im}{\mathfrak{m}}
\newcommand{\ilim}{\varinjlim}
\newcommand{\eps}{\varepsilon}
\newcommand{\assf}{\ass^{\rm f}}
\newcommand{\NN}{\mathbf{N}}

\DeclareMathOperator{\ass}{Ass}
\DeclareMathOperator{\supp}{Supp}
\DeclareMathOperator{\var}{Var}
\DeclareMathOperator{\zd}{ZD}
\DeclareMathOperator{\spec}{Spec}
\DeclareMathOperator{\card}{Card}
\DeclareMathOperator{\coker}{Coker}
\DeclareMathOperator{\ke}{Ker}
\DeclareMathOperator{\ob}{Ob}
\DeclareMathOperator{\idem}{Idem}

\newdir{ >}{{}*!/-10pt/\dir{>}}
\newdir{ (}{!/5pt/@{ }*@^{(}}

\begin{document}

\title{Torsion functors, small or large}
\author{Fred Rohrer}
\address{Grosse Grof 9, 9470 Buchs, Switzerland}
\email{fredrohrer@math.ch}

\begin{abstract}
Let $\ia$ be an ideal in a commutative ring $R$. For an $R$-module $M$, we consider the small $\ia$-torsion $\ga_\ia(M)=\{x\in M\mid\exists n\in\N:\ia^n\subseteq(0:_Rx)\}$ and the large $\ia$-torsion $\oga_\ia(M)=\{x\in M\mid\ia\subseteq\sqrt{(0:_Rx)}\}$. This gives rise to two functors $\ga_\ia$ and $\oga_\ia$ that coincide if $R$ is noetherian, but not in general. In this article, basic properties of as well as the relation between these two functors are studied, and several examples are presented, showing that some well-known properties of torsion functors over noetherian rings do not generalise to non-noetherian rings.
\end{abstract}

\maketitle\thispagestyle{fancy}


\section*{Introduction}

Throughout the following, let $R$ be a ring and let $\ia\subseteq R$ be an ideal. (Rings are always understood to be commutative.) The $\ia$-torsion functor and its right derived cohomological functor (i.e., local cohomology) are useful tools in commutative algebra and algebraic geometry. They behave nicely if the ring $R$ is noetherian, and they are mostly studied in this situation (cf. \cite{bs} for a comprehensive treatment from an algebraic point of view). If we wish to extend the theory of local cohomology to non-noetherian rings, then we face several challenges. For example, torsion functors need not preserve injectivity of modules (\cite{qr}, \cite{ss}), and local cohomology need not be isomorphic to \v{C}ech cohomology (\cite{lipman}, \cite{ss}). But even more fundamental, the two definitions of torsion functors that are usually used and that are equivalent over a noetherian ring turn out to be not equivalent anymore in a non-noetherian setting. Namely, for an $R$-module $M$ we consider its sub-$R$-modules \[\ga_\ia(M)=\{x\in M\mid\exists n\in\N:\ia^n\subseteq(0:_Rx)\}\] and \[\oga_\ia(M)=\{x\in M\mid\ia\subseteq\sqrt{(0:_Rx)}\}.\] Then, $\ga_\ia(M)\subseteq\oga_\ia(M)$, but these two modules need not be equal. These definitions can be extended to two preradicals $\ga_\ia$ and $\oga_\ia$; we call the first one \textit{the small $\ia$-torsion functor} and the second one \textit{the large $\ia$-torsion functor,} reflecting the aforementioned inclusion between them. It is the aim of this article to study and compare these two functors over arbitrary rings. Its goal is twofold. First, it provides a collection of basic results on torsion functors over arbitrary rings; considering the ubiquity of noetherian hypotheses in the literature, this seems useful. Second, it shall serve as a warning: When leaving the cosy noetherian home for the non-noetherian wilderness, one has to be very careful, since one may lose properties of torsion functors one got used to.\smallskip

\chead{\footnotesize Torsion functors, small or large}
We briefly mention now some of the questions we will discuss.\smallskip

\begin{aufz}
\item[A)] \textit{For an ideal $\ib\subseteq R$, how can we compare $\ia$-torsion and $\ib$-torsion functors?}
\end{aufz}
It was shown in \cite{tor1} that the comparison of $\ga_\ia$ and $\ga_\ib$ may be delicate, and in particular that $\ga_\ia$ and $\ga_{\sqrt{\ia}}$ need not coincide. Large torsion functors behave better: $\oga_\ia=\oga_\ib$ if and only if $\sqrt{\ia}=\sqrt{\ib}$.\smallskip

\begin{aufz}
\item[B)] \textit{When do we have $\ga_\ia=\oga_\ia$?}
\end{aufz}
This holds if $R$ is noetherian, but not in general. We show that $\ga_\ia=\oga_\ia$ if and only if $\ia$ is half-centred, i.e., if an $R$-module $M$ fulfills $\ga_\ia(M)=M$ if and only if its weak assassin is contained in $\var(\ia)$.\smallskip

\begin{aufz}
\item[C)] \textit{When are the $\ia$-torsion functors radicals?}
\end{aufz}
The functor $\oga_\ia$ is a radical, and hence $\ga_\ia$ is a radical in case $R$ is noetherian. In general, $\ga_\ia$ need not be a radical. It is moreover possible that $\ga_\ia$ is a radical without coinciding with $\oga_\ia$.\smallskip

\begin{aufz}
\item[D)] \textit{When do the $\ia$-torsion $R$-modules form a Serre class?}
\end{aufz}
The $R$-modules $M$ with $\oga_\ia(M)=M$ form a Serre class, but those with $\ga_\ia(M)=M$ need not do so.\smallskip

\begin{aufz}
\item[E)] \textit{When do $\ia$-torsion functors commute with flat base change?}
\end{aufz}
If $\ia$ is of finite type, then $\ga_\ia=\oga_\ia$ commutes with flat base change. In general this need not be the case. On the positive side, $\ga_\ia$ commutes with flat base change to an $R$-algebra whose underlying $R$-module is Mittag-Leffler.\smallskip

\begin{aufz}
\item[F)] \textit{What exactness properties do $\ia$-torsion functors have?}
\end{aufz}
The functors $\ga_\ia$ and $\oga_\ia$ are left exact and commute with direct sums. In general, they are not exact, and they commute neither with infinite products nor with right filtering inductive limits.\smallskip

The first two sections are of a preliminary nature. Section 1 contains some results on nilpotency and idempotency, for in the examples and counterexamples in later sections we often consider ideals that are nil or idempotent, and we also give some specific results for these cases. In Section 2 we collect some basics on Mittag-Leffler modules, a notion which pops up in our investigation of the flat base change property. The study of torsion functors starts with Section 3.\smallskip

In a subsequent work (\cite{part2}) we will extend the results on the interplay of assassins and weak assassins with small torsion functors from \cite{asstor} and generalise them to large torsion functors. A further step should be an investigation of the relation between small and large local cohomology, i.e., the right derived cohomological functors of the small and the large torsion functor. For this, a study similar to \cite{qr} about the behaviour of injective modules under large torsion functors will be necessary.\bigskip

\noindent\textbf{Notation.} We denote by $\catmod(R)$ the category of $R$-modules. For a set $I$ we denote by $R[(X_i)_{i\in I}]$ the polynomial algebra in the indeterminates $(X_i)_{i\in I}$ over $R$. We denote by $\idem(R)$ the set of idempotent elements of $R$, by $\var(\ia)$ the set of prime ideals of $R$ containing $\ia$, and by $\sqrt{\ia}$ the radical of $\ia$. For an $R$-module $M$ we denote by $\zd_R(M)$ the set of zerodivisors on $M$, by $\supp_R(M)$ the support of $M$, by $\ass_R(M)$ the assassin of $M$, and by $\assf_R(M)$ the weak assassin of $M$. If $h\colon R\rightarrow S$ is a morphism of rings and no confusion can arise, then we denote by $\ia S$ the ideal $\langle h(\ia)\rangle_S\subseteq S$ and by $\bullet\res_R\colon\catmod(S)\rightarrow\catmod(R)$ the scalar restriction functor with respect to $h$. In general, notation and terminology follow Bourbaki's \textit{\'El\'ements de math\'ematique.}


\section{Preliminaries on nilpotency and idempotency}

\begin{no}\label{nil10}
The ideal $\ia$ is called \textit{nilpotent} if there exists $n\in\N$ such that $\ia^n=0$, \textit{quasinilpotent} (or \textit{nil of bounded index}) if there exists $n\in\N$ such that for every $r\in\ia$ we have $r^n=0$, \textit{T-nilpotent} if for every sequence $(r_i)_{i\in\N}$ in $\ia$ there exists $n\in\N$ such that $\prod_{i=0}^nr_i=0$, and \textit{nil} if for every $r\in\ia$ there exists $n\in\N$ such that $r^n=0$.
\end{no}

\begin{no}\label{cora95}
We say that \textit{$\ia$ has a power of finite type} if there exists $n\in\N^*$ such that $\ia^n$ is of finite type. Ideals of finite type and nilpotent ideals clearly have powers of finite type, but an ideal that has a power of finite type need neither be of finite type nor nilpotent (\cite[Section 3]{ghr}).
\end{no}

\begin{no}\label{nil15}
It is readily checked that \[\xymatrix@R15pt{\ia\text{ is nilpotent}\ar@{=>}[r]\ar@{=>}[d]&\ia\text{ is quasinilpotent}\ar@{=>}[d]\\\ia\text{ is T-nilpotent}\ar@{=>}[r]&\ia\text{ is nil},}\] and that these four properties are equivalent if $\ia$ has a power of finite type. The zero ideal has all of these properties, while -- provided $R\neq 0$ -- the ideal $R$ has none of them.
\end{no}

\begin{no}\label{nil20}
A) A quasinilpotent ideal need not be T-nilpotent. Indeed, let $K$ be a field of characteristic $2$, let \[R\dfgl K[(X_i)_{i\in\N}]/\langle X_i^2\mid i\in\N\rangle,\] let $Y_i$ denote the canonical image of $X_i$ in $R$ for $i\in\N$, and let $\ia\dfgl\langle Y_i\mid i\in\N\rangle_R$. For $r\in\ia$ there exist $n\in\N$ and $(a_i)_{i=0}^n\in K$ with $r=\sum_{i=0}^na_iY_i$, implying $r^2=\sum_{i=0}^na_i^2Y_i^2=0$. Therefore, $\ia$ is quasinilpotent. Furthermore, $\prod_{i=0}^nY_i\neq 0$ for every $n\in\N$, hence $\ia$ is not T-nilpotent. Note that $R$ is a $0$-dimensional local ring with maximal ideal $\ia$. (This example is taken from \cite[III.4.4]{lazard}.)\smallskip

B) A T-nilpotent ideal need not be quasinilpotent. Indeed, let $K$ be a field, let \[R\dfgl K[(X_i)_{i\in\N}]/\langle\{X_iX_j\mid i,j\in\N,i\neq j\}\cup\{X_i^{i+1}\mid i\in\N\}\rangle,\] let $Y_i$ denote the canonical image of $X_i$ in $R$ for $i\in\N$, and let $\ia\dfgl\langle Y_i\mid i\in\N\rangle_R$. If there exists a sequence $(r_i)_{i\in\N}$ in $\ia$ such that for every $n\in\N$ we have $\prod_{i=0}^nr_i\neq 0$, then there exists such a sequence where all the $r_i$ are monomials in $(Y_j)_{j\in\N}$, and as $Y_iY_j=0$ for $i,j\in\N$ with $i\neq j$, there exists $k\in\N$ such that all the $r_i$ are monomials in $Y_k$, implying the contradiction $\prod_{i=0}^{k+1}r_i=0$. Therefore, $\ia$ is T-nilpotent. Furthermore, $Y_n^n\neq 0$ for every $n\in\N$, hence $\ia$ is not quasinilpotent. Note that $R$ is a $0$-dimensional local ring with maximal ideal $\ia$. (This example is taken from \cite[2.12]{qr}.)\smallskip

C) It follows from A) and B) that the only implications between the four properties in \ref{nil15} are the ones given there and their obvious composition.
\end{no}

\begin{no}\label{nil25}
A quasinilpotent and T-nilpotent ideal need not be nilpotent. Indeed, let $K$ be a field of characteristic $2$, let \[R\dfgl K[(X_i)_{i\in\N}]/\langle\{X_i^2\mid i\in\N\}\cup\{X_iX_j\mid i,j\in\N,2i<j\}\rangle,\] let $Y_i$ denote the canonical image of $X_i$ in $R$ for $i\in\N$, and let $\ia\dfgl\langle Y_i\mid i\in\N\rangle_R$. If $n\in\N$, then for $i,j\in[n+1,2n]$ we have $2i\geq j$, implying $0\neq\prod_{i=n+1}^{2n}Y_i\in\ia^n$. Thus, $\ia$ is not nilpotent. Moreover, the same argument as in \ref{nil20} A) shows that $\ia$ is quasinilpotent. Finally, if there exists a sequence $(r_i)_{i\in\N}$ in $\ia$ such that for every $n\in\N$ we have $\prod_{i=0}^nr_i\neq 0$, then there exists such a sequence where all the $r_i$ are monomials in $(Y_j)_{j\in\N}$, and thus we can suppose that $r_i=Y_{j_i}$ for every $i\in\N$, where $(j_i)_{i\in\N}$ is an injective family in $\N$. Injectivity of $(j_i)_{i\in\N}$ implies that there exists $k\in\N$ with $j_k>2j_0$, yielding the contradiction $\prod_{i=0}^kr_i=0$. Therefore, $\ia$ is T-nilpotent. Note that $R$ is a $0$-dimensional local ring with maximal ideal $\ia$. (This example was pointed out by Yves Cornulier.)
\end{no}

\begin{no}\label{idem05}
A) The ideal $\ia$ is said to be \textit{generated by a single idempotent} if there exists $e\in\idem(R)$ with $\ia=\langle e\rangle_R$, \textit{generated by idempotents} if there exists $E\subseteq\idem(R)$ with $\ia=\langle E\rangle_R$, and \textit{idempotent} if $\ia^2=\ia$.\smallskip

B) The ideal $\ia$ is generated by a single idempotent if and only if the canonical injection $\ia\hookrightarrow R$ is a section in $\catmod(R)$ (\cite[I.8.11 Proposition 10]{a}).
\end{no}

\begin{no}\label{idem20}
A) Clearly, we have \[\xymatrix{\parbox{3.5cm}{$\ia$ is generated by\\ a single idempotent}\ar@{=>}[r]&\parbox{3cm}{$\ia$ is generated\\ by idempotents}\ar@{=>}[r]&\parbox{3cm}{$\ia$ is idempotent.}}\]

B) If $\ia$ is of finite type, then the three properties in A) are equivalent. Indeed, if $\ia$ is of finite type and idempotent, then by Nakayama's Lemma (\cite[Theorem 2.2]{matsumura}) there exists $x\in(0:_R\ia)$ with $1-x\in\ia$, hence $(1-x)^2=1-x-x(1-x)=1-x$, and therefore $1-x\in\idem(R)$. For $y\in\ia$ we have $y=y-xy=y(1-x)$, hence $y\in\langle 1-x\rangle_R$. It follows $\ia\subseteq\langle 1-x\rangle_R\subseteq\ia$. So, $\ia=\langle 1-x\rangle_R$ is generated by a single idempotent.\smallskip

C) None of the converses in A) holds. Indeed, let $K$ be a field, let $R\dfgl K[(X_i)_{i\in\Z}]$, and consider the ideals $\ia\dfgl\langle X_i\mid i\in\Z\rangle_R$, $\ib\dfgl\langle X_i^2-X_i\mid i\in\Z\rangle_R$ and $\ic\dfgl\langle X_i^2-X_{i+1}\mid i\in\Z\rangle_R$. The canonical image of $\ia$ in $R/\ib$ is generated by idempotents, but not by a single idempotent, and the canonical image of $\ia$ in $R/\ic$ is idempotent, but not generated by idempotents.\smallskip

D) The ring $R/\ib$ from C) is absolutely flat. Indeed, let $Y_i$ denote the canonical image of $X_i$ in $R/\ib$ for $i\in\Z$. If $\ip\in\spec(R)$ and $i\in\N$, then $Y_i(1-Y_i)=0\in\ip$, hence $Y_i\in\ip$ or $1-Y_i\in\ip$. This shows \[\spec(R)=\{\langle Y_i\mid i\in I\rangle_R+\langle 1-Y_i\mid i\in J\rangle_R\mid (I,J)\text{ is a partition of }\N\}.\] So, if $\ip,\iq\in\spec(R)$ with $\ip\subsetneq\iq$, then we get the contradiction \[\{i\in\N\mid Y_i\in\ip\}\subsetneqq\{i\in\N\mid Y_i\in\iq\},\quad\N\setminus\{i\in\N\mid Y_i\in\ip\}\subsetneqq\N\setminus\{i\in\N\mid Y_i\in\iq\},\] thus $\dim(R)=0$. It is readily checked that $R/\ib$ is reduced and hence absolutely flat.\smallskip

E) The ring $R/\ic$ from C) is a $1$-dimensional Bezout domain. (Recall that a \textit{Bezout domain} is a domain whose ideals of finite type are principal, and that filtering unions of Bezout domains are Bezout domains (\cite[VII.1 Exercice 20]{ac}).) Indeed, let $Y_i$ denote the canonical image of $X_i$ in $R/\ic$ for $i\in\Z$. If $n\in\Z$, then $K[Y_n]$ is a polynomial algebra in one indeterminate over $K$, hence a principal ideal domain, and we have $K[Y_n]\subseteq K[Y_{n-1}]$. So, $R/\ic=\bigcup_{n\in\N}K[Y_{-n}]$ is a filtering union of principal ideal domains and therefore a Bezout domain. Now, we consider the extension $K[Y_0]\subseteq R/\ic$. For $i\in\N$ we have $Y_{-i}^{2^i}-Y_0=0$, hence $Y_{-i}$ is integral over $K[Y_0]$. Therefore, the extension $K[Y_0]\subseteq R/\ic$ is integral, and thus $\dim(R/\ic)=\dim(K[Y_0])=1$ (\cite[VIII.3 Th\'eor\`eme 1]{ac}) as claimed. (This is essentially \cite[Example 39]{hutchins}.)
\end{no}

\begin{prop}\label{idem30}
The following statements are equivalent: (i) Every ideal of finite type of $R$ is generated by a single idempotent; (ii) Every ideal of $R$ is generated by idempotents; (iii) Every ideal of $R$ is idempotent; (iv) $R$ is absolutely flat.
\end{prop}

\begin{proof}
``(i)$\Rightarrow$(ii)$\Rightarrow$(iii)'' is clear. ``(iii)$\Rightarrow$(i)'' follows from \ref{idem20} B). ``(i)$\Rightarrow$(iv)'': If $\ia\subseteq R$ is an ideal of finite type, then the canonical injection $\ia\hookrightarrow R$ is a section (\ref{idem05} B)), hence the induced morphism $\ia\otimes_RM\rightarrow R\otimes_RM$ is a monomorphism for every $R$-module $M$ (\cite[I.2.1 Lemme 2]{ac}), and thus $R$ is absolutely flat (\cite[I.2.3 Remarque 1]{ac}). ``(iv)$\Rightarrow$(iii)'': If $\ia\subseteq R$ is an ideal, then tensoring the exact sequence of $R$-modules $0\rightarrow\ia\hookrightarrow R\rightarrow R/\ia$ with the flat $R$-module $R/\ia$ yields an exact sequence of $R$-modules $0\rightarrow\ia/\ia^2\rightarrow R/\ia\overset{\Id_{R/\ia}}\longrightarrow R/\ia$, and thus $\ia=\ia^2$ is idempotent. (This argument is essentially \cite[1.10]{qr}.) 
\end{proof}


\section{Preliminaries on Mittag-Leffler modules}

\noindent\textit{Throughout this section, let $M$ be an $R$-module.}

\begin{no}\label{mittag10}
A) For a family $\NN=(N_i)_{i\in I}$ of $R$-modules there is a canonical morphism of $R$-modules \[\eps_{M,\NN}\colon M\otimes_R\bigl(\prod_{i\in I}N_i\bigr)\rightarrow\prod_{i\in I}(M\otimes_RN_i)\] with $\eps_{M,\NN}(x\otimes(y_i)_{i\in I})=(x\otimes y_i)_{i\in I}$ for $x\in M$ and $(y_i)_{i\in I}\in\prod_{i\in I}N_i$, which need neither be a mono- nor an epimorphism (\cite[059I]{stacks}).\smallskip

B) Let $I$ be a set. For an $R$-module $N$ we consider the constant family $\NN=(N)_{i\in I}$ and get by A) a canonical morphism of $R$-modules \[\eps_{M,N,I}\dfgl\eps_{M,\NN}\colon M\otimes_R(N^I)\rightarrow(M\otimes_RN)^I,\] which need neither be a mono- nor an epimorphism. (It need not be a monomorphism by the first example in \cite[059I]{stacks} together with \ref{mittag30} below, and it need not be an epimorphism by the second example in \loccit\!\!)
\end{no}

\begin{no}\label{mittag14}
A) The canonical morphism $\eps_{M,\NN}$ is an isomorphism for every $\NN$ if and only if $M$ is of finite presentation, and it is an epimorphism for every $\NN$ if and only if $M$ is of finite type (\cite[059K, 059J]{stacks}).\smallskip

B) The $R$-module $M$ is called \textit{Mittag-Leffler} if the canonical morphism $\eps_{M,\NN}$ is a monomorphism for every $\NN$. (By \cite[059M]{stacks} this is equivalent to the usual definition of Mittag-Leffler modules (cf. \cite[0599]{stacks}).) If $M$ is of finite presentation, then it is Mittag-Leffler by A). If $M$ is projective, then it is Mittag-Leffler by \cite[059R]{stacks}.
\end{no}

\begin{no}\label{mittag20}
Let $\NN=(N_i)_{i\in I}$ be a family of $R$-modules. Taking the products of the canonical injections $\iota_i\colon N_i\rightarrowtail\bigoplus_{j\in I}N_j$ for $i\in I$ and keeping in mind \cite[II.1.5 Proposition 5]{a} we get a canonical monomorphism of $R$-modules \[\Delta_{\NN}\colon\prod_{i\in I}N_i\rightarrowtail\prod_{i\in I}\bigoplus_{j\in I}N_j.\]
\end{no}

\begin{prop}\label{mittag30}
Let $M$ be flat and let $I$ be a set. The following statements are equivalent:
\begin{equi}
\item[(i)] $\eps_{M,\NN}$ is a monomorphism for every family $\NN=(N_i)_{i\in I}$ of $R$-modules;
\item[(ii)] $\eps_{M,N,I}$ is a monomorphism for every $R$-module $N$.
\end{equi}
\end{prop}

\begin{proof}
Clearly, (i) implies (ii). Conversely, suppose that (ii) holds, let $\NN=(N_i)_{i\in I}$ be a family of $R$-modules, let $N\dfgl\bigoplus_{i\in I}N_i$, and let $M\otimes_R\NN\dfgl(M\otimes_RN_i)_{i\in I}$. There is a commutative diagram of $R$-modules \[\xymatrix@C40pt{\prod_{i\in I}(M\otimes_RN_i)\ar@{ >->}[r]^(.44){\Delta_{M\otimes_R\NN}}&\prod_{i\in I}\bigoplus_{j\in I}(M\otimes_RN_j)\ar[r]^(.48){\cong}&\prod_{i\in I}(M\otimes_R(\bigoplus_{j\in I}N_j))\\M\otimes_R(\prod_{i\in I}N_i)\ar[rr]^{M\otimes_R\Delta_{\NN}}\ar[u]^{\eps_{M,\NN}}&&M\otimes_R(\prod_{i\in I}\bigoplus_{j\in I}N_j),\ar@{ >->}[u]_{\eps_{M,N,I}}}\] where the unmarked isomorphism is the canonical one (\cite[II.3.7 Proposition 7]{a}). As $\Delta_{\NN}$ is a monomorphism (\ref{mittag20}) and $M$ is flat, it follows that $M\otimes_R\Delta_{\NN}$ is a monomorphism, hence $\eps_{M,\NN}$ is a monomorphism, and thus (i) holds.
\end{proof}

\begin{cor}\label{mittag40}
A flat $R$-module $M$ is Mittag-Leffler if and only if $\eps_{M,N,I}$ is a monomorphism for every $R$-module $N$ and every set $I$.
\end{cor}

\begin{proof}
Immediately from \ref{mittag30}.
\end{proof}

\begin{prop}\label{mittag50}
Let $R$ be absolutely flat. The following statements are equivalent:
\begin{equi}
\item[(i)] $M$ is Mittag-Leffler;
\item[(ii)] There exists a cardinal $\kappa\geq\card(R)$ such that $\eps_{M,R,\kappa}$ is a monomorphism;
\item[(iii)] $M$ is pseudocoherent.
\end{equi}
\end{prop}

\begin{proof}
By \cite[Theorem 1 Corollary]{goodearl}, (i) and (ii) are equivalent, and they are equivalent to sub-$R$-modules of finite type of $M$ being projective. Since an $R$-module of finite type is projective if and only if it is of finite presentation (\cite[X.1.5 Proposition 8 Corollaire]{a}), this last condition is equivalent to $M$ being pseudocoherent (i.e., sub-$R$-modules of $M$ of finite type are of finite presentation).
\end{proof}


\section{Definition of torsion functors}

\begin{no}\label{def10}
A \textit{preradical} is a subfunctor of $\Id_{\catmod(R)}$. A \textit{radical} is a preradical $F$ such that $F(M/F(M))=0$ for every $R$-module $M$. Being subfunctors of the $R$-linear functor $\Id_{\catmod(R)}$, preradicals are $R$-linear.
\end{no}

\begin{no}\label{def20}
For an $R$-module $M$ we set \[\ga_\ia(M)\dfgl\{x\in M\mid\exists n\in\N:\ia^n\subseteq(0:_Rx)\}\] and \[\oga_\ia(M)\dfgl\{x\in M\mid\ia\subseteq\sqrt{(0:_Rx)}\},\] thus obtaining sub-$R$-modules \[\ga_\ia(M)\subseteq\oga_\ia(M)\subseteq M.\] By restriction and coastriction, a morphism of $R$-modules $h\colon M\rightarrow N$ induces morphisms of $R$-modules \[\ga_\ia(h)\colon\ga_\ia(M)\rightarrow\ga_\ia(N)\quad\text{and}\quad\oga_\ia(h)\colon\oga_\ia(M)\rightarrow\oga_\ia(N).\] So, there are subfunctors \[\ga_\ia\hookrightarrow\oga_\ia\hookrightarrow\Id_{\catmod(R)}.\] In particular, $\ga_\ia$ and $\oga_\ia$ are preradicals and hence $R$-linear (\ref{def10}). We call $\ga_\ia$ \textit{the small $\ia$-torsion functor} and $\oga_\ia$ \textit{the large $\ia$-torsion functor.}
\end{no}

\begin{no}\label{def30}
A) If $M$ is an $R$-module, then \[\oga_\ia(M)=\{x\in M\mid\supp_R(\langle x\rangle_R)\subseteq\var(\ia)\}.\] Indeed, if $x\in M$, then $\supp_R(\langle x\rangle_R)=\var(0:_Rx)$ (\cite[II.4.4 Proposition 17]{ac}), and thus we have $\ia\subseteq\sqrt{(0:_Rx)}$ if and only if $\supp_R(\langle x\rangle_R)\subseteq\var(\ia)$ (\cite[II.4.3 Proposition 11 Corollaire 2]{ac}).\smallskip

B) It follows from A) and \cite[II.4.4 Proposition 16]{ac} that for an $R$-module $M$ we have $M=\oga_\ia(M)$ if and only if $\supp_R(M)\subseteq\var(\ia)$.\smallskip

C) By \cite[1.2.11]{bs} (in whose proof noetherianness of $R$ is not used), there is a canonical isomorphism of functors \[\ga_\ia(\bullet)\cong\ilim_{n\in\N}\hm{R}{R/\ia^n}{\bullet}\] from $\catmod(R)$ to itself.\smallskip

D) For an ideal $\ib\subseteq R$, it is readily checked that \[\ga_{\ia+\ib}=\ga_\ia\cap\ga_\ib=\ga_\ia\circ\ga_\ib\quad\text{and}\quad\oga_{\ia+\ib}=\oga_\ia\cap\oga_\ib=\oga_\ia\circ\oga_\ib.\]
\end{no}

\begin{no}\label{def50}
Let $\ib\subseteq R$ be an ideal. Then, $\ga_\ia$ is a subfunctor of $\ga_\ib$ if and only if there exists $n\in\N$ with $\ib^n\subseteq\ia$. In particular, $\ga_\ia=\ga_{\ia^n}$ for every $n\in\N^*$. If $\ga_\ia=\ga_\ib$, then $\sqrt{\ia}=\sqrt{\ib}$, but the converse need not hold; it does hold if there exist $m,n\in\N$ with $\sqrt{\ia}^m\subseteq\ia$ and $\sqrt{\ib}^n\subseteq\ib$, hence in particular if $\sqrt{\ia}$ and $\sqrt{\ib}$ are of finite type. Similarly, we have $\ga_\ia=\ga_{\sqrt{\ia}}$ if and only if there exists $n\in\N$ with $\sqrt{\ia}^n\subseteq\ia$, but this need not hold in general (\cite[2.2, 2.3, 3.4]{tor1}).
\end{no}

\begin{prop}\label{def40}
Let $\ib\subseteq R$ be an ideal.

a) $\oga_\ia$ is a subfunctor of $\oga_\ib$ if and only if $\ib\subseteq\sqrt{\ia}$.

b) $\oga_\ia=\oga_\ib$ if and only if $\sqrt{\ia}=\sqrt{\ib}$.

c) $\oga_\ia=\oga_{\sqrt{\ia}}$.
\end{prop}

\begin{proof}
a) If $\oga_\ia$ is a subfunctor of $\oga_\ib$, then $R/\ia=\oga_\ia(R/\ia)\subseteq\oga_\ib(R/\ia)$, and therefore $\ib\subseteq\sqrt{(0:_R1+\ia)}=\sqrt{\ia}$. Conversely, suppose that $\ib\subseteq\sqrt{\ia}$. If $M$ is an $R$-module and $x\in\oga_\ia(M)$, then $\ia\subseteq\sqrt{(0:_Rx)}$, hence $\ib\subseteq\sqrt{\ia}\subseteq\sqrt{(0:_Rx)}$, and therefore $x\in\oga_\ib(M)$. b) and c) follow immediately from a) and b), resp.
\end{proof}

\begin{prop}\label{def80}
a) $\ga_\ia=0$ if and only if $\oga_\ia=0$ if and only if $\ia=R$.

b) $\ga_\ia=\Id_{\catmod(R)}$ if and only if $\ia$ is nilpotent.

c) $\oga_\ia=\Id_{\catmod(R)}$ if and only if $\ia$ is nil.
\end{prop}

\begin{proof}
a) follows from the facts that $\ga_\ia$ is a subfunctor of $\oga_\ia$ and that $\ga_\ia(R/\ia)=R/\ia$. b) and c) follow by considering $\ga_\ia(R)$ and $\oga_\ia(R)$, resp.
\end{proof}

\begin{no}\label{cora120}
A) If $h\colon R\rightarrow S$ is a morphism of rings, then \[\ga_\ia(\bullet\res_R)=\ga_{\ia S}(\bullet)\res_R\quad\text{and}\quad\oga_\ia(\bullet\res_R)=\oga_{\ia S}(\bullet)\res_R\] as functors from $\catmod(S)$ to $\catmod(R)$.\smallskip

B) Let $\ib\subseteq R$ be an ideal. Considering the canonical morphism of rings $R\rightarrow R/\ib$ and setting $\overline{\ia}\dfgl(\ia+\ib)/\ib\subseteq R/\ib$, we get from A) \[\ga_\ia(\bullet\res_R)=\ga_{\overline{\ia}}(\bullet)\res_R\quad\text{and}\quad\oga_\ia(\bullet\res_R)=\oga_{\overline{\ia}}(\bullet)\res_R\] as functors from $\catmod(R/\ib)$ to $\catmod(R)$.
\end{no}


\section{Coincidence}

\begin{no}\label{cora105}
If $\ia$ is nil but not nilpotent, then $\ga_\ia\neq\oga_\ia$ (\ref{def80}); concrete such examples will be discussed in \ref{cora158}. As there exists ideals that are idempotent and nil, but not nilpotent (e.g. the maximal ideal of the $0$-dimensional local ring $T$ constructed in \cite[2.2]{qr}) it follows that even if $\ia$ is idempotent, then $\ga_\ia$ and $\oga_\ia$ need not coincide; a concrete such example with an idempotent non-nil ideal will be discussed in \ref{cora159} B).
\end{no}

\begin{prop}\label{cora180}
For an $R$-module $M$ we have \[\xymatrix{\ga_\ia(M)=M\ar@{=>}[r]&\assf_R(M)\subseteq\var(\ia)\ar@{<=>}[r]&\oga_\ia(M)=M.}\]
\end{prop}

\begin{proof}
The first implication holds by \cite[4.2]{asstor}. If $\assf_R(M)\subseteq\var(\ia)$ and $x\in M$, then $\ia\subseteq\sqrt{(0:_Rx)}$, hence $x\in\oga_\ia(M)$, implying $\oga_\ia(M)=M$. Conversely, suppose that $\oga_\ia(M)=M$ and let $\ip\in\assf_R(M)$. Then, there exists $x\in M$ with $\ip\in\min(0:_Rx)$. For $r\in\ia$ there exists $n\in\N$ with $r^nx=0$, hence $r^n\in(0:_Rx)\subseteq\ip$, yielding $r\in\ip$, and therefore $\ip\in\var(\ia)$.
\end{proof}

\begin{no}\label{cora190}
The first implication in \ref{cora180} need not be an equivalence. If it is an equivalence for every $R$-module $M$, then $\ia$ is called \textit{half-centred} (\cite[4.3]{asstor}). The notion of half-centredness will be put in a larger framework in \cite{part2}.
\end{no}

\begin{prop}\label{cora195}
a) $\ia$ is half-centred if and only if $\ia^n$ is so for some $n\in\N^*$.

b) If $\ib\subseteq R$ is a half-centred ideal with $\ia\subseteq\ib\subseteq\sqrt{\ia}$, then $\ia$ is half-centred.

c) If $\ia$ has a power of finite type, then it is half-centred.

d) If $\ia$ is generated by idempotents, then it is half-centred.
\end{prop}

\begin{proof}
a) holds since $\var(\ia)=\var(\ia^n)$ and $\ga_\ia=\ga_{\ia^n}$ (\ref{def50}). b) holds since $\var(\ia)=\var(\ib)$ and $\ga_\ib$ is a subfunctor of $\ga_\ia$ (\ref{def50}). c) follows from a) and \cite[4.3 C)]{asstor}. d) Let $M$ be an $R$-module with $\assf_R(M)\subseteq\var(\ia)$. If $x\in M$, then $\ia\subseteq\bigcap\assf_R(M)\subseteq\bigcap\min(0:_Rx)=\sqrt{(0:_Rx)}$, and as $\ia$ is generated by idempotents it follows $\ia\subseteq(0:_Rx)$, hence $x\in\ga_\ia(M)$.
\end{proof}

\begin{thm}\label{cora200}
We have $\ga_\ia=\oga_\ia$ if and only if $\ia$ is half-centred.
\end{thm}

\begin{proof}
If $\ga_\ia=\oga_\ia$, then $\ia$ is half-centred by \ref{cora180}. Conversely, suppose that $\ia$ is half-centred. Let $M$ be an $R$-module. If $x\in\oga_\ia(M)$, then $\oga_\ia(\langle x\rangle_R)=\langle x\rangle_R$, so \ref{cora180} implies $\ga_\ia(\langle x\rangle_R)=\langle x\rangle_R$ and therefore $x\in\ga_\ia(M)$. It follows $\oga_\ia(M)=\ga_\ia(M)$ and thus the claim.
\end{proof}

\begin{cor}\label{cora100}
a) If $\ia$ has a power of finite type or is generated by idempotents, then $\ga_\ia=\oga_\ia$.

b) If $R$ is noetherian or absolutely flat, then $\ga_\ia=\oga_\ia$.

c) If $\ib\subseteq R$ is an ideal with $\ia\subseteq\ib\subseteq\sqrt{\ia}$ and $\ga_\ib=\oga_\ib$, then $\ga_\ia=\oga_\ia$.

d) If $M$ is a noetherian $R$-module, then $\ga_\ia(M)=\oga_\ia(M)$.
\end{cor}

\begin{proof}
a) and c) follow from \ref{cora200} and \ref{cora195}. b) follows from a) and \ref{idem30}. d) follows from \ref{cora120} B) and b) by considering $M$ as a module over the noetherian ring $R/(0:_RM)$ (\cite[VIII.1.3 Proposition 6 Corollaire]{a}).
\end{proof}

\begin{no}\label{cora115}
A) If $\ib\subseteq R$ is an ideal with $\ia\subseteq\ib\subseteq\sqrt{\ia}$ that has a power of finite type, then $\ga_\ia=\oga_\ia$ by \ref{cora100} a) and c). This hypothesis is strictly weaker than $\ia$ having a power of finite type. In fact, there exists an ideal without powers of finite type whose radical is principal. Indeed, let $p$ be a prime number, let \[S\dfgl\Z[(X_i)_{i\in\N}]/\langle pX_{i+1}-X_i\mid i\in\N\rangle,\] let $Y_i$ denote the canonical image of $X_i$ in $S$ for $i\in\N$, let $\ic\dfgl\langle Y_i\mid i\in\N\rangle_S$, let $\im\dfgl\langle p\rangle_S+\ic$, and let $\ia\dfgl\langle p^2\rangle_S+\ic$. Then, $R\dfgl S_\im$ is a $2$-dimensional valuation ring whose maximal ideal $\im_\im=\langle p\rangle_R$ is principal, and $\spec(R)=\{0,\ic_\im,\im_\im\}$ (\cite[2.9]{qr}). Clearly, $\ia_\im\subseteq\im_\im$. If $\ia_\im\subseteq\ic_\im$, then $p^2\in\ic_\im$, and as $\ic_\im$ is prime it follows that $p\in\ic_\im$, yielding the contradiction $\ic_\im=\im_\im$; therefore, $\ia_\im\not\subseteq\ic_\im$. This implies that $\var(\ia_\im)=\{\im_\im\}$, thus $\sqrt{\ia_\im}=\im_\im$ is principal.

We assume now that there exists $n\in\N^*$ such that $\ia_\im^n$ is of finite type. Then, there exists $k\in\N$ such that $\ia_\im^n$ is generated by 
\[\bigcup_{j=1}^n\{p^{2j}\cdot\prod_{l=1}^{n-j}Y_{i_l}\mid i_1,\ldots,i_{n-j}\in[0,k]\}.\] Now we consider $Y_{k+1}^n\in\ia_\im^n$. There exist $h\in S\setminus\im$, $f\in S$ and $g\in\langle Y_i\mid i\in[0,k]\rangle_S$ with $hY_{k+1}^r=p^2f+g$. We write the elements $h$, $f$ and $g$ of $S$ as classes of polynomials in indeterminates whose indices are at least $k+1$, and then compare the coefficients of $Y_{k+1}^n$ in these two classes of polynomials in $S$. On the left side, since $h\notin\im$, this coefficient is not a multiple of $p$. On the right side, the coefficients of $Y_{k+1}^n$ in $fp^2$ and in $g$ are multiples of $p$. This is a contradiction, and thus no power of $\ia_\im$ is of finite type.\smallskip

B) If $\ib\subseteq R$ is an ideal generated by idempotents with $\ia\subseteq\ib\subseteq\sqrt{\ia}$, then $\ia=\ib$. Therefore, combining \ref{cora100} a) and c) yields no further criterion for $\ga_\ia$ and $\oga_\ia$ to coincide.
\end{no}


\section{Radicality}

\begin{prop}\label{cora140}
a) The functor $\oga_\ia$ is a radical.

b) If $\ia$ is half-centred or idempotent, then $\ga_\ia$ is a radical.
\end{prop}

\begin{proof}
a) Let $M$ be an $R$-module, and let $x\in M$ with $x+\oga_\ia(M)\in\oga_\ia(M/\oga_\ia(M))$. For $r\in\ia$ there exists $n\in\N$ with $r^nx\in\oga_\ia(M)$, hence there exists $m\in\N$ with $r^mx=0$. This implies $x\in\oga_\ia(M)$ and thus the claim. b) If $\ia$ is half-centred, then this is clear from a) and \ref{cora200}. Suppose that $\ia$ is idempotent. Let $M$ be an $R$-module, and let $x\in M$ with $x+\ga_\ia(M)\in\ga_\ia(M/\ga_\ia(M))$. Then, $\ia x\subseteq\ga_\ia(M)$, hence $\ia x=\ia^2x=0$, thus $x\in\ga_\ia(M)$, and so we get the claim.
\end{proof}

\begin{prop}\label{cora150}
Let $\ib\subseteq R$ be an ideal, and let $\overline{\ia}\dfgl(\ia+\ib)/\ib\subseteq R/\ib$. If $\ga_\ia$ is a radical, then so is $\ga_{\overline{\ia}}$.
\end{prop}

\begin{proof}
If $M$ is an $R/\ib$-module, then \ref{cora120} B) implies \[\ga_{\overline{\ia}}(M/\ga_{\overline{\ia}}(M))\res_R=\ga_\ia(M\res_R/\ga_\ia(M\res_R))=0\] and hence we get the claim.
\end{proof}

\begin{lemma}\label{cora155}
If $\ib\subseteq R$ is an ideal with $(\ia+\ib)/\ib\subseteq\ga_\ia(R/\ib)\neq R/\ib$, then $\ga_\ia((R/\ib)/\ga_\ia(R/\ib))\neq 0$.
\end{lemma}

\begin{proof}
As $\ga_\ia(R/\ib)\neq R/\ib$ we have $(1+\ib)+\ga_\ia(R/\ib)\in((R/\ib)/\ga_\ia(R/\ib))\setminus 0$. Moreover, $\ia((1+\ib)+\ga_\ia(R))=(\ia+\ib)+\ga_\ia(R/\ib)=0\in(R/\ib)/\ga_\ia(R/\ib)$, hence $(1+\ib)+\ga_\ia(R/\ib)\in\ga_\ia((R/\ib)/\ga_\ia(R/\ib))\setminus 0$.
\end{proof}

\begin{no}\label{cora158}
A) Let $R$ and $\ia$ be as in \ref{nil20} A). Let $\ib\dfgl\sum_{i\in\N}\ia^iY_i\subseteq R$. If $i\in\N$, then $\ia^iY_i\subseteq\ib$, hence $Y_i+\ib\in\ga_\ia(R/\ib)$, and thus $\ia/\ib\subseteq\ga_\ia(R/\ib)$. Moreover, $\ia^i$ contains every squarefree monomial with $i$ factors, hence $\ia^i\not\subseteq\ib$, because in every squarefree monomial with $i$ factors in $\ib$ occurs $Y_{i-1}$. It follows that $\ga_\ia(R/\ib)\neq R/\ib$, and hence $\ga_\ia$ is not a radical by \ref{cora155}. (More precisely, since $\ia$ is the unique maximal ideal of $R$ we get $\ga_\ia(R/\ib)=\ia/\ib$, and therefore $\ga_\ia((R/\ib)/\ga_\ia(R/\ib))=\ga_\ia(R/\ia)=R/\ia$.) (This example is due to Pham Hung Quy.)\smallskip

B) Let $R$ and $\ia$ be as in \ref{nil20} B). Clearly, $\ga_\ia(R)=\ia$, so \ref{cora155} (applied with $\ib=0$) implies $\ga_\ia(R/\ga_\ia(R))\neq 0$, and thus $\ga_\ia$ is not a radical.\smallskip

C) Let $R$ and $\ia$ be as in \ref{nil25}. Let $i\in\N$. Let $J\subseteq\N$ be a subset of cardinality $2i+2$, so that we may write $J=\{m_1,\ldots,m_{2i+2}\}$ with $m_j<m_{j+1}$ for every $j\in[1,2i+1]$. Then, $2i<m_{2i+2}$, hence $Y_{m_{2i+2}}Y_i=0$, and therefore $(\prod_{j\in J}Y_j)Y_i=0$. It follows that $\ia^{2i+2}Y_i=0$, and therefore $Y_i\in\ga_\ia(R)$. This shows $\ia\subseteq\ga_\ia(R)$. As $\ia$ is the unique maximal ideal of $R$ but not nilpotent, we get $\ga_\ia(R)=\ia$. As in B) it follows that $\ga_\ia$ is not a radical.\smallskip

D) Let $K$ be a field, let \[R\dfgl K[(X_i)_{i\in\N}]/\langle X_i^{i+1}\mid i\in\N\rangle,\] let $Y_i$ denote the canonical image of $X_i$ in $R$ for $i\in\N$, and let $\ia\dfgl\langle Y_i\mid i\in\N\rangle$. Factoring out $\ib\dfgl\langle Y_iY_j\mid i,j\in\N,i\neq j\rangle$, it follows from B) and \ref{cora150} that $\ga_\ia$ is not a radical. (For a direct proof on use of \ref{cora155} we may observe that $\ia/\ib\subseteq\ga_\ia(R/\ib)\neq R/\ib$.) (Note that $R$ is a $0$-dimensional local ring whose maximal ideal $\ia$ is neither quasinilpotent nor T-nilpotent.)
\end{no}

\begin{no}\label{cora159}
A) Let $R$, $\ia$ and $\ib$ be as in \ref{idem20} C). Then, $\ia/\ib\subseteq R/\ib$ is generated by idempotents, hence $\ga_{\ia/\ib}=\oga_{\ia/\ib}$ is a radical (\ref{cora100} a), \ref{cora140} a)).\smallskip

B) Let $R$, $\ia$ and $\ic$ be as in \ref{idem20} C). Let $\overline{R}\dfgl R/\ic$, let $\overline{\ia}\dfgl\ia/\ic$, and let $Y_i$ denote the canonical image of $X_i$ in $\overline{R}$ for $i\in\Z$. Then, $\overline{\ia}\subseteq\overline{R}$ is idempotent, hence $\ga_{\overline{\ia}}$ is a radical (\ref{cora140} b)). We consider the $\overline{R}$-module $M\dfgl\overline{R}/\langle Y_0\rangle_{\overline{R}}$. For $i\in\N$ we have $Y_iM=0$ and $Y_{-i}^{2^i}=Y_{-i+i}=Y_0$, hence $Y_{-i}^{2^i}M=0$, implying $\oga_{\overline{\ia}}(M)=M$. We assume now that there exists $g\in\overline{R}\setminus\langle Y_0\rangle_{\overline{R}}$ with $\overline{\ia}g\in\langle Y_0\rangle_{\overline{R}}$. Then, there occurs in $g$ a monomial $t$ that is not a multiple of $Y_0$, but such that $Y_it$ is a multiple of $Y_0$ for every $i\in\Z_{<0}$. So, there exists a strictly increasing sequence $(l_j)_{j=0}^k$ in $\Z_{<0}$ with $t=\prod_{j=0}^kY_{l_j}$, and choosing $i<l_0$ we get the contradiction that $Y_i\cdot\prod_{j=0}^kY_{l_j}$ is a multiple of $Y_0$. This shows that $\ga_{\overline{\ia}}(M)=0$, and thus it follows that $\ga_{\overline{\ia}}\neq\oga_{\overline{\ia}}$.
\end{no}

\begin{no}\label{cora160}
A) We saw in \ref{cora159} B) that $\ga_\ia$ may be a radical but $\ga_\ia\neq\oga_\ia$. Therefore, $\oga_\ia$ need not be the smallest radical containing $\ga_\ia$ in the sense of \cite[VI.1.5]{sten}.\smallskip

B) It may happen even for a nil ideal that $\ga_\ia$ is a radical but $\ga_\ia\neq\oga_\ia$. Indeed, if $\ia$ is idempotent, nil and nonzero, then $\ga_\ia$ is a radical (\ref{cora140} b)), but $\ga_\ia(R)=\linebreak(0:_R\ia)\neq R=\oga_\ia(R)$ (\ref{def80} c)), hence $\ga_\ia\neq\oga_\ia$. An example of a $0$-dimensional local ring whose maximal ideal fulfils these conditions is given in \cite[2.2]{qr}.
\end{no}

\begin{qu}\label{cora170}
The preceding results give rise to the following question:
\begin{aufz}
\item[$(*)$] \textit{For which ideals $\ia$ is $\ga_\ia$ a radical?}
\end{aufz}
\end{qu}


\section{Serre classes}

\begin{no}\label{serre10}
A class $\mathcal{S}$ of $R$-modules is called a \textit{Serre class} if $0\in\mathcal{S}$ and for every exact sequence $0\rightarrow N\rightarrow M\rightarrow P\rightarrow 0$ of $R$-modules we have $M\in\mathcal{S}$ if and only if $N,P\in\mathcal{S}$.
\end{no}

\begin{no}\label{serre20}
We set \[\mathcal{T}_\ia\dfgl\{M\in\ob(\catmod(R))\mid\ga_\ia(M)=M\}\] and \[\overline{\mathcal{T}}_\ia\dfgl\{M\in\ob(\catmod(R))\mid\oga_\ia(M)=M\}.\] Clearly, $0\in\mathcal{T}_\ia\subseteq\overline{\mathcal{T}}_\ia$.
\end{no}

\begin{prop}\label{serre30}
The class $\overline{\mathcal{T}}_\ia$ is a Serre class.
\end{prop}

\begin{proof}
This follows immediately from \ref{def30} B) and \cite[II.4.4 Proposition 16]{ac}.
\end{proof}

\begin{prop}\label{serre40}
Let $0\rightarrow N\overset{h}\rightarrow M\overset{l}\rightarrow P\rightarrow 0$ be an exact sequence of $R$-modules.

a) If $M\in\mathcal{T}_\ia$, then $N,P\in\mathcal{T}_\ia$.

b) If $N$ is noetherian and $N,P\in\mathcal{T}_\ia$, then $M\in\mathcal{T}_\ia$.
\end{prop}

\begin{proof}
a) Suppose that $M\in\mathcal{T}_\ia$. If $x\in N$, then $h(x)\in\ga_\ia(M)$, hence there exists $n\in\N$ with $h(\ia^nx)=\ia^nh(x)=0$, thus with $\ia^nx=0$, and therefore $x\in\ga_\ia(N)$. If $x\in P$, then there exists $y\in\ga_\ia(M)$ with $l(y)=x$, hence there exists $n\in\N$ with $\ia^ny=0$, implying $\ia^nx=\ia^nl(y)=l(\ia^ny)=0$, and thus $x\in\ga_\ia(P)$.

b) Suppose that $N$ is noetherian and $N,P\in\mathcal{T}_\ia$. Let $x\in M$. Then, there exists $n\in\N$ with $l(\ia^nx)=\ia^nl(x)=0$, implying $\ia^nx\subseteq h(N)$. So, there exists a sub-$R$-module $U\subseteq N$ with $h(U)=\ia^nx$. Since $U$ is of finite type, there exists $m\in\N$ with $\ia^mU=0$, implying $\ia^{m+n}x=\ia^mh(U)=h(\ia^mU)=h(0)=0$, and thus $x\in\ga_\ia(M)$.
\end{proof}

\begin{prop}\label{serre60}
If $\ia$ is half-centred or idempotent, then $\mathcal{T}_\ia$ is a Serre class.
\end{prop}

\begin{proof}
If $\ia$ is half-centred, then this follows from \ref{cora200} and \ref{serre30}. Let $\ia$ be idempotent, and let $0\rightarrow N\overset{h}\rightarrow M\overset{l}\rightarrow P\rightarrow 0$ be an exact sequence of $R$-modules with $N,P\in\mathcal{T}_\ia$. If $x\in M$, then $\ia l(x)=0$, hence $\ia x\subseteq h(N)$, thus $\ia x=\ia^2x\subseteq\ia h(N)=h(\ia N)=h(0)=0$, and therefore $x\in\ga_\ia(M)$. The claim follows now from \ref{serre40} a).
\end{proof}

\begin{no}\label{serre70}
In general, $\mathcal{T}_\ia$ need not be a Serre class, not even if $\ia$ is quasinilpotent and T-nilpotent. Indeed, if $0\neq\ia=\ga_\ia(R)\neq R$ (e.g. as in \ref{nil25} (cf. \ref{cora158} C))), then $\ga_\ia(\ia)=\ia$ and $\ga_\ia(R/\ia)=R/\ia$, but $\ga_\ia(R)\neq R$.
\end{no}

\begin{qu}\label{serre80}
The preceding results give rise to the following question:
\begin{aufz}
\item[$(*)$] \textit{For which ideals $\ia$ is $\mathcal{T}_\ia$ a Serre class?}
\end{aufz}
\end{qu}


\section{Bounded torsion and zerodivisors}

\begin{no}\label{bounded10}
A) Let $M$ be an $R$-module. We say that $M$ is \textit{of strongly bounded large $\ia$-torsion} if there exists $n\in\N$ with $\ia^nM\cap\oga_\ia(M)=0$, and \textit{of strongly bounded small $\ia$-torsion} if there exists $n\in\N$ with $\ia^nM\cap\ga_\ia(M)=0$. Furthermore, we say that $M$ is \textit{of bounded large $\ia$-torsion} if there exists $n\in\N$ with $\ia^n\oga_\ia(M)=0$, and \textit{of bounded small $\ia$-torsion} if there exists $n\in\N$ with $\ia^n\ga_\ia(M)=0$.\smallskip

B) It is readily checked that \[\xymatrix@C50pt@R15pt{\parbox{4.5cm}{\center$M$ is of strongly\\bounded large $\ia$-torsion}\ar@{=>}[r]\ar@{=>}[d]&\parbox{4.5cm}{\center$M$ is of strongly\\bounded small $\ia$-torsion}\ar@{=>}[d]\\\parbox{4.5cm}{\center$M$ is of bounded\\large $\ia$-torsion}\ar@{=>}[r]&\parbox{4.5cm}{\center$M$ is of bounded\\small $\ia$-torsion.}}\]

C) All four properties in A) are inherited by sub-$R$-modules.\smallskip

D) The $R$-module $M$ is of bounded large $\ia$-torsion if and only if $\oga_\ia(M)$ is of bounded large $\ia$-torsion if and only if $\oga_\ia(M)$ is of strongly bounded large $\ia$-torsion. Similarly, $M$ is of bounded small $\ia$-torsion if and only if $\ga_\ia(M)$ is of bounded small $\ia$-torsion if and only if $\ga_\ia(M)$ is of strongly bounded small $\ia$-torsion.
\end{no}

\begin{no}\label{bounded20}
A) If $n\in\N$, then the $R$-module $R/\ia^n$ is of strongly bounded large $\ia$-torsion.\smallskip

B) Noetherian $R$-modules are of strongly bounded large $\ia$-torsion by \ref{cora100} d) and the Artin--Rees Lemma.\smallskip

C) If $\ia$ is nilpotent, then every $R$-module is of strongly bounded large $\ia$-torsion; we will see in \ref{bounded45} b) that the converse need not hold.\smallskip

D) In general, an $R$-module need not be of bounded small $\ia$-torsion, not even if $\ia$ is quasinilpotent and T-nilpotent or if $R$ is noetherian. Indeed, if $R$ and $\ia$ are as in \ref{nil20} B) or \ref{nil25}, then $\ia$ is not nilpotent and $\ga_\ia(R)=\ia$ (\ref{cora158} B), C)), hence no power of $\ia$ annihilates $\ga_\ia(R)$. Similarly, if $R$, $\ia$ and $\ib$ are as in \ref{nil20} A) or \ref{cora158} D), then no power of $\ia$ is contained in $\ib$ and $\ga_\ia(R/\ib)=\ia/\ib$ (\ref{cora158} A), D)), hence no power of $\ia$ annihilates $\ga_\ia(R/\ib)$. For an example with a noetherian ring, let $R$ be a noetherian integral domain that is not a field, let $\ia\neq 0$, and let $M\dfgl\bigoplus_{n\in\N}R/\ia^n$. Then, $\ga_\ia(M)=M$, and for $n\in\N$ we have $\ia^n\neq\ia^{n+1}$ (\ref{idem20} B)), hence $\ia^n(R/\ia^{n+1})\neq 0$, and therefore $\ia^n\ga_\ia(M)\neq 0$.
\end{no}

\begin{no}\label{bounded30}
A) An $R$-module of strongly bounded small $\ia$-torsion need not be of bound\-ed large $\ia$-torsion. Indeed, if $\ia$ is idempotent, nil, nonzero and fulfils $(0:_R\ia)=0$, then $\ga_\ia(R)=0\neq R=\oga_\ia(R)$ (\ref{cora160} B)), hence $R$ is of strongly bounded small $\ia$-torsion, but not of bounded large $\ia$-torsion. The proof of \cite[3.9]{asstor} shows that the maximal ideal of the $0$-dimensional local ring $T$ constructed in \cite[2.2]{qr} fulfils these conditions. (A further such example will be given in \ref{bounded60}.)\smallskip

B) An $R$-module of bounded large $\ia$-torsion need not be of strongly bounded small $\ia$-torsion, not even if $R$ is noetherian. Indeed, let $R\dfgl K[X,Y]$ be the polynomial algebra in two indeterminates over a field $K$, and let $\ia\dfgl\langle X,Y\rangle_R$. Note that $\ga_\ia=\oga_\ia$ (\ref{cora100} a)) and $\ga_\ia(\ia)=0$. We consider the free $R$-module $F$ with basis $(e_i)_{i\in\N}$ and the $R$-module \[M\dfgl F/\langle\{Xe_0,Ye_0\}\cup\{Xe_{2i+1}+Ye_{2i+2}-e_0\mid i\in\N\}\rangle_R.\] Note that $\ia M=M$. For $i\in\N$ we denote by $a_i$ the canonical image of $e_i$ in $M$, and consider the sub-$R$-module $M_i\dfgl\langle a_j\mid j\in[0,2i]\rangle_R\subseteq M$ and the canonical inclusion $f_i\colon M_i\hookrightarrow M_{i+1}$. Then, $M=\bigcup_{i\in\N}M_i$. Note that $M_0=\langle a_0\rangle_R$, hence $\ia M_0=0$, and therefore $\ga_\ia(M_0)=M_0$ and $\ia\ga_\ia(M_0)=0$. It follows that $\ga_\ia(M)\neq 0$, and as $\ia M=M$ we see that $M$ is not of strongly bounded small $\ia$-torsion.

For $i\in\N$ there is an isomorphism of $R$-modules $\coker(f_i)\overset{\cong}\longrightarrow\ia$ with $a_j\mapsto 0$ for $j\in[0,2i]$, $a_{2i+1}\mapsto Y$ and $a_{2i+2}\mapsto -X$, implying $\ga_\ia(\coker(f_i))=0$ and therefore $\ga_\ia(M_i)=\ga_\ia(M_{i+1})$. Inductively, we get $\ga_\ia(M_i)=M_0$ for every $i\in\N$, and thus it follows that $\ga_\ia(M)=M_0$ (directly or by \ref{exact20}). Now, $\ia\ga_\ia(M)=\ia M_0=0$, and therefore $M$ is of bounded large $\ia$-torsion. (This example was pointed out by Will Sawin.)\smallskip

C) It follows from A) and B) that the only implications between the four properties in \ref{bounded10} B) are the ones given there and their obvious composition.
\end{no}

\begin{no}\label{bounded40}
A) If an $R$-module $M$ is of bounded large $\ia$-torsion, then $\ga_\ia(M)=\oga_\ia(M)$. The converse need not hold. Indeed, let $R$ and $\ia$ be such that there exists an $R$-module $N$ that is not of bounded small $\ia$-torsion (e.g. as in \ref{bounded20} D)), and let $M\dfgl\ga_\ia(N)$. Then, $\ga_\ia(M)=M$, hence $\ga_\ia(M)=\oga_\ia(M)$. For every $n\in\N$ we have $\ia^n\ga_\ia(N)\neq 0$, hence $\ia^n\oga_\ia(M)\neq 0$, and thus $M$ is not of bounded large $\ia$-torsion.\smallskip

B) An $R$-module of bounded large $\ia$-torsion and of strongly bounded small $\ia$-torsion is of strongly bounded large $\ia$-torsion, as follows from A).
\end{no}

\begin{prop}\label{bounded45}
a) If $\ia$ is idempotent, then every $R$-module is of bounded small $\ia$-torsion, but not necessarily of bounded large $\ia$-torsion or of strongly boun\-ded small $\ia$-torsion.

b) If $\ia$ is generated by idempotents, then every $R$-module is of strongly bounded large $\ia$-torsion.
\end{prop}

\begin{proof}
a) The first statement is clear from the definition, and the second one follows from \ref{bounded30} A). For the third one, let $K$ be a field, let $Q$ denote the additive monoid of positive rationals, let $K[Q]$ be the algebra of $Q$ over $K$, let $\{e_r\mid r\in Q\}$ denote its canonical basis, let $\im\dfgl\langle e_r\mid r>0\rangle_{K[Q]}$, let $R\dfgl K[Q]/\langle e_r\mid r>1\rangle_{K[Q]}$, and let $\overline{\cdot}\colon K[Q]\rightarrow R$ denote the canonical projection. Then, $\overline{\im}$ is an idempotent ideal of $R$ and $(0:_R\overline{\im})=\langle\overline{e_1}\rangle_R$, implying $\im\cap(0:_R\im)=\langle\overline{e_1}\rangle_R\neq 0$, and thus the $R$-module $R$ is not of strongly bounded small $\im$-torsion. (This example is due to Pham Hung Quy.)

b) Let $E\subseteq\idem(R)$ with $\ia=\langle E\rangle_R$. Let $M$ be an $R$-module. For $x\in\linebreak\ia M\cap\oga_\ia(M)$ there exists a finite subset $F\subseteq E$ with $x\in\langle F\rangle_RM\cap\oga_\ia(M)\subseteq\langle F\rangle_RM\cap\oga_{\langle F\rangle_R}(M)$ (\ref{def40} a)). So, keeping in mind \ref{idem20} B), we can suppose that $\ia=\langle e\rangle_R$ with $e\in\idem(R)$. Now, if $x\in\ia M\cap(0:_M\ia)$, then there exists $y\in M$ with $x=ey$, implying $0=ex=eey=ey=x$. This shows $\ia M\cap\oga_\ia(M)=\ia M\cap(0:_M\ia)=0$ (\ref{cora100} a)), and thus the claim is proven.
\end{proof}

\begin{no}\label{bounded60}
Let $\overline{R}$, $\overline{\ia}$ and $M$ be as in \ref{cora159} B) (cf. \ref{idem20} C)). Then, $\overline{\ia}$ is idempotent, and thus $M$ is of bounded small $\overline{\ia}$-torsion (\ref{bounded45} a)). Moreover, $0=\ga_{\overline{\ia}}(M)\neq\oga_{\overline{\ia}}(M)=M$ (\ref{cora159} B)) and $\overline{\ia}\not\subseteq\langle Y_0\rangle_{\overline{R}}$, so that $M$ is of strongly bounded small $\overline{\ia}$-torsion but not of bounded large $\overline{\ia}$-torsion.
\end{no}

\begin{no}\label{bounded49}
If $\ia$ is nil, then an $R$-module is of strongly bounded large $\ia$-torsion if and only if it is of bounded large $\ia$-torsion (\ref{def80} c)). The analogous statement for small $\ia$-torsion does not hold, as can be seen by noting that the ideal $\overline{\im}$ in the proof of \ref{bounded45} a) is nil.
\end{no}

\begin{no}\label{bounded70}
A) Let $M$ be an $R$-module. Then, we clearly have \[\xymatrix{\ga_\ia(M)\neq 0\;\ar@{=>}[r]&\;\oga_\ia(M)\neq 0\;\ar@{=>}[r]&\;\ia\subseteq\zd_R(M).}\]

B) If $\ass_R(M)$ is finite and $\zd_R(M)=\bigcup\ass_R(M)$, hence in particular if $M$ is noetherian, then the implications in A) are equivalences (\cite[2.1.1]{bs}).\smallskip

C) None of the converses in A) holds. Examples with $\ga_\ia(M)=0$ and $\oga_\ia(M)\neq 0$ were given in \ref{bounded30} A) and in \ref{bounded60}. For a counterexample to the converse of the second implication (even with a noetherian ring) we consider the factorial ring $R\dfgl\Z[X]$. Let $P$ denote the set of prime elements of $R$, let $M\dfgl\bigoplus_{p\in P}R/\langle p\rangle_R$, let $p,q\in P$ with $p\neq q$ and $\ia\dfgl\langle p,q\rangle_R\neq R$. Then, every non-unit of $R$ is a zerodivisor on $M$, hence $\ia\subseteq\zd_R(M)$, and $\oga_\ia(M)=\ga_\ia(M)=0$ (\ref{cora100} a)).
\end{no}


\section{Flat base change}

\noindent\textit{Throughout this section, let $h\colon R\rightarrow S$ be a flat morphism of rings.}

\begin{no}\label{flat10}
A) Composing the exact functor $S\otimes_R\bullet$ from $\catmod(R)$ to $\catmod(S)$ with \[\ga_\ia\hookrightarrow\oga_\ia\hookrightarrow\Id_{\catmod(R)}\] yields monomorphisms \[S\otimes_R\ga_\ia(\bullet)\rightarrowtail S\otimes_R\oga_\ia(\bullet)\rightarrowtail S\otimes_R\bullet\] of functors from $\catmod(R)$ to $\catmod(S)$.\smallskip

B) Let $M$ be an $R$-module. Let $m\in\oga_\ia(M)$ and $s\in S$. If $r\in\ia$, then $1\otimes r\in \ia S$, and there exists $n\in\N$ with $r^nm=0$; we get $(1\otimes r)^ns\otimes x=s\otimes(r^nm)=0$. It follows that $s\otimes m\in\oga_{\ia S}(S\otimes_RM)$. Therefore, the monomorphism $S\otimes_R\oga_\ia(M)\rightarrowtail S\otimes_RM$ induces by coastriction a monomorphism of $S$-modules \[\overline{\rho}_\ia^h(M)\colon S\otimes_R\oga_\ia(M)\rightarrowtail\oga_{\ia S}(S\otimes_RM).\]

C) Let $M$ be an $R$-module. Let $m\in\ga_\ia(M)$ and $s\in S$. There exists $n\in\N$ with $\ia^nm=0$, hence $(\ia S)^n(s\otimes m)=s\otimes(\ia^nm)=0$. It follows that $s\otimes m\in\ga_{\ia S}(S\otimes_RM)$. Therefore, the monomorphism $S\otimes_R\ga_\ia(M)\rightarrowtail S\otimes_RM$ induces by coastriction a monomorphism of $S$-modules \[\rho_\ia^h(M)\colon S\otimes_R\ga_\ia(M)\rightarrowtail\ga_{\ia S}(S\otimes_RM).\]

D) The monomorphisms $\overline{\rho}_\ia^h(M)$ and $\rho_\ia^h(M)$ in B) and C) are natural in $M$. Thus, we obtain a commutative diagram \[\xymatrix@R18pt{
S\otimes_R\ga_\ia(\bullet)\ar@{ >->}[r]\ar@{ >->}[d]_{\rho_\ia^h}&S\otimes_R\oga_\ia(\bullet)\ar@{ >->}[d]^{\overline{\rho}_\ia^h}\\
\ga_{\ia S}(S\otimes_R\bullet)\ar@{ (->}[r]&\oga_{\ia S}(S\otimes_R\bullet)
}\] of functors from $\catmod(R)$ to $\catmod(S)$.
\end{no}

\begin{no}\label{flat15}
If $S\subseteq R$ is a subset and $\eta\colon R\rightarrow S^{-1}R$ denotes the canonical epimorphism of rings, then, as a special case of \ref{flat10}, we get canonical monomorphisms of functors \[\rho^\eta_\ia\colon S^{-1}\ga_\ia(\bullet)\rightarrowtail\ga_{S^{-1}\ia}(S^{-1}\bullet)\quad\text{ and }\quad\overline{\rho}^\eta_\ia\colon S^{-1}\oga_\ia(\bullet)\rightarrowtail\oga_{S^{-1}\ia}(S^{-1}\bullet).\]
\end{no}

\begin{no}\label{flat20}
A) It follows from \ref{def80}, \ref{cora100} a) and \ref{flat10} that $\overline{\rho}_\ia^h$ is an isomorphism if $\ia$ is nil, and that $\rho_\ia^h=\overline{\rho}_\ia^h$ is an isomorphism if $\ia$ is nilpotent.\smallskip

B) In general, neither $\rho_\ia^h$ nor $\overline{\rho}_\ia^h$ need be an isomorphism, not even if $\ia$ is generated by idempotents and $h$ is the canonical epimorphism of a ring of fractions obtained by inverting a single element. Indeed, let $K$ be a field, let $R\dfgl K[(X_i)_{i\in\N}]/\langle X_i^2-X_i\mid i\in\N^*\rangle$, let $Y_i$ denote the canonical image of $X_i$ in $R$ for $i\in\N$, let $\ia\dfgl\langle Y_i\mid i\in\N^*\rangle_R$, let $\ib\dfgl\langle Y_0^iY_i\mid i\in\N^*\rangle_R$, and let $M\dfgl R/\ib$. We consider the subset $S\dfgl\{Y_0^i\mid i\in\N\}\subseteq R$ and the canonical epimorphism of rings $\eta\colon R\rightarrow S^{-1}R$. Then, $\ia\subseteq R$ and $S^{-1}\ia\subseteq S^{-1}R$ are generated by idempotents, hence $\rho_\ia^\eta=\overline{\rho}_\ia^\eta$ (\ref{cora100} a)). Moreover, $\ga_\ia(M)=0$, hence $S^{-1}\ga_\ia(M)=0$. The canonical epimorphism $M\twoheadrightarrow R/\ia$ induces an isomorphism $S^{-1}M\overset{\cong}\longrightarrow S^{-1}R/S^{-1}\ia$. Therefore, $\hm{S^{-1}R}{S^{-1}R/S^{-1}\ia}{S^{-1}M}\neq 0$, hence $\ga_{S^{-1}\ia}(S^{-1}M)\neq 0$ (\ref{def30} C)). It follows that $\rho_\ia^\eta(M)=\overline{\rho}_\ia^\eta(M)$ is not an isomorphism.
\end{no}

\begin{no}\label{flat30}
A) There is a canonical monomorphism \[\sigma^h_{\ia}\colon S\otimes_R\hm{R}{R/\ia}{\bullet}\rightarrowtail\hm{S}{S/\ia S}{S\otimes_R\bullet}\] of functors from $\catmod(R)$ to $\catmod(S)$; it is an isomorphism if $\ia$ is of finite type (\cite[I.2.10 Proposition 11]{ac}).\smallskip

B) Applying A) to $\ia^n$ for every $n\in\N$, taking inductive limits of the morphisms $\sigma^h_{\ia^n}$ thus obtained for $n\in\N$, and keeping in mind \ref{def30} C) and \cite[II.6.3 Proposition 7 Corollaire 3]{a}, we get a monomorphism \[S\otimes_R\ga_\ia(\bullet)\rightarrowtail\ga_{\ia S}(S\otimes_R\bullet)\] of functors from $\catmod(R)$ to $\catmod(S)$, equal to the monomorphism $\rho_\ia^h$ in \ref{flat10}. Clearly, if $\sigma^h_{\ia^n}(M)$ is an isomorphism for infinitely many $n$ and some $R$-module $M$, then $\rho^h_\ia(M)$ is an isomorphism, and if $\sigma^h_{\ia^n}$ is an isomorphism for infinitely many $n$, then $\rho^h_\ia$ is an isomorphism.
\end{no}

\begin{prop}\label{flat40}
Let $\kappa$ be a cardinal such that infinitely many powers of $\ia$ have generating sets of cardinality at most $\kappa$. Let $M$ be an $R$-module such that $\eps_{S,M,\kappa}\colon S\otimes_R(M^\kappa)\rightarrow(S\otimes_RM)^\kappa$ is a monomorphism. Then, the canonical monomorphism of $S$-modules \[\rho_\ia^h(M)\colon S\otimes_R\ga_\ia(M)\rightarrowtail\ga_{\ia S}(S\otimes_RM)\] is an isomorphism.
\end{prop}

\begin{proof}
Let $\ib$ be a power of $\ia$ with a generating set of cardinality at most $\kappa$. Then, there is an exact sequence of $R$-modules \[R^{\oplus\kappa}\rightarrow R\rightarrow R/\ib\rightarrow 0,\] hence an exact sequence of $S$-modules \[S^{\oplus\kappa}\rightarrow S\rightarrow S/\ib S\rightarrow 0.\] Applying $S\otimes_R\hm{R}{\bullet}{M}$ to the first and $\hm{S}{\bullet}{S\otimes_RM}$ to the second sequence yields a commutative diagram of $S$-modules with exact rows \[\xymatrix{0\ar[r]&S\otimes_R\hm{R}{R/\ib}{M}\ar[r]\ar[d]^{\sigma_\ib^h(M)}&S\otimes_RM\ar[r]\ar@{=}[d]&S\otimes_R\hm{R}{R^{\oplus\kappa}}{M}\ar[d]^{\tau}\\0\ar[r]&\hm{S}{S/\ib S}{S\otimes_RM}\ar[r]&S\otimes_RM\ar[r]&\hm{S}{S^{\oplus\kappa}}{S\otimes_RM}.}\] There are canonical isomorphisms \[S\otimes_R\hm{R}{R^{\oplus\kappa}}{M}\cong S\otimes(M^\kappa)\] and \[\hm{S}{S^{\oplus\kappa}}{S\otimes_RM}\cong(S\otimes_RM)^\kappa,\] so that, up to isomorphism, $\tau$ is equal to the monomorphism $\eps_{S,M,\kappa}$. So, $\sigma^h_\ib(M)$ is an isomorphism by the Five Lemma, and \ref{flat30} B) yields the claim.
\end{proof}

\begin{cor}\label{flat42}
a) If $\ia$ has a power of finite type, then $\rho^h_\ia=\overline{\rho}^h_\ia$ is an isomorphism.

b) If $M$ is a noetherian $R$-module, then $\rho^h_\ia(M)=\overline{\rho}^h_\ia(M)$ is an isomorphism.

c) If the $R$-module $S$ is Mittag-Leffler, then $\rho^h_\ia$ is an isomorphism.
\end{cor}

\begin{proof}
a) follows from \ref{cora100} a) and \ref{flat30}. b) follows from \ref{cora120} B) and a) by considering $M$ as a module over the noetherian ring $R/(0:_RM)$ (\cite[VIII.1.3 Proposition 6 Corollaire]{a}). c) follows from \ref{mittag40} and \ref{flat40}.
\end{proof}

\begin{no}\label{flat50}
A) If the $R$-module $S$ is projective, hence in particular if it is free or of finite presentation, then it is Mittag-Leffler (\ref{mittag14} B)), and thus $\rho^h_\ia$ is an isomorphism by \ref{flat42} c).\smallskip

B) If $S=R[M]$ is the algebra of a commutative monoid $M$ over $R$, or if $R$ is perfect, then $\rho^h_\ia$ is an isomorphism by A). (Recall that a ring is \textit{perfect} if its flat modules are projective.)\smallskip

C) If $R$ is absolutely flat and the $R$-module $S$ is pseudocoherent, then $\rho^h_\ia=\overline{\rho}^h_\ia$ is an isomorphism by \ref{mittag50}, \ref{cora100} b) and \ref{flat42} c).\smallskip

D) Let $S\subseteq R$ be a subset, let $\eta\colon R\rightarrow S^{-1}R$ be the canonical epimorphism of rings, and let $M$ be an $R$-module with $\zd_R(M)=\{0\}$. Then, $\eps_{S^{-1}R,M,\kappa}$ is a monomorphism for every cardinal $\kappa$ (\cite[3.1]{webb}), and thus $\rho^\eta_\ia(M)$ is an isomorphism by \ref{flat40}.\smallskip

E) If $\rho^h_\ia$ is an isomorphism for every ideal $\ia$ (e.g., if $R$ is noetherian (\ref{flat42} a))), then the $R$-module $S$ need not be Mittag-Leffler. Indeed, the canonical bimorphism of rings $\Z\rightarrow\Q$ yields a counterexample (\cite[059U]{stacks}).
\end{no}

\begin{qu}\label{flat60}
The preceding results give rise to the following question:
\begin{aufz}
\item[$(*)$] \textit{For which $h$ and which $\ia$ is $\rho^h_\ia$ or $\overline{\rho}^h_\ia$ an isomorphism?}
\end{aufz}
\end{qu}


\section{Exactness properties}

\begin{no}\label{exact10}
A) Let $F$ be a preradical. By \cite[VI.1.7]{sten}, the following statements are equivalent: (i) $F$ is left exact; (ii) if $M$ is an $R$-module and $N\subseteq M$ is a sub-$R$-module, then $F(N)=F(M)\cap N$; (iii) $F$ is idempotent, and if $M$ is an $R$-module and $N\subseteq F(M)$ is a sub-$R$-module, then $N=F(N)$.\smallskip

B) Let $F$ be a left exact preradical. Then, $F$ commutes with unions and with intersections. Indeed, let $I$ be a set, let $M$ be an $R$-module, and let $(M_i)_{i\in I}$ be a family of sub-$R$-modules of $M$. For $y\in F(\bigcup_{i\in I}M_i)\subseteq\bigcup_{i\in I}M_i$ there exists $j\in I$ with $y\in M_j$, so A) implies $y\in M_j\cap F(\bigcup_{i\in I}M_i)=F(M_j)$ and hence $y\in\bigcup_{i\in I}F(M_i)$. Therefore, $F$ commutes with unions. Applying A) twice yields $\bigcap_{i\in I}F(M_i)=\bigcap_{i\in I}(F(M)\cap M_i)=F(M)\cap\bigcap_{i\in I}M_i=F(\bigcap_{i\in I}M_i)$. Therefore, $F$ commutes with intersections.
\end{no}

\begin{prop}\label{exact20}
The functors $\ga_\ia$ and $\oga_\ia$ are left exact and idempotent, and they commute with unions, intersections, and direct sums.
\end{prop}

\begin{proof}
Both preradicals satisfy condition (ii) in \ref{exact10} A), hence are left exact and idempotent, and thus commute with unions and intersections by \ref{exact10} B). Concerning the claim about direct sums, it is readily checked directly that for any family $(M_i)_{i\in I}$ of $R$-modules, the canonical morphisms \[\bigoplus_{i\in I}\ga_\ia(M_i)\rightarrow\ga_\ia\bigl(\bigoplus_{i\in I}M_i\bigr)\quad\text{and}\quad\bigoplus_{i\in I}\oga_\ia(M_i)\rightarrow\oga_\ia\bigl(\bigoplus_{i\in I}M_i\bigr)\] are isomorphisms.
\end{proof}

\begin{no}\label{exact30}
By general nonsense and \ref{exact20}, the functors $\ga_\ia$ and $\oga_\ia$ commute with projective limits if and only if they commute with infinite products, and they commute with inductive limits if and only if they are right exact.\smallskip
\end{no}

\begin{prop}\label{exact40}
a) If $\ia$ is nilpotent, then $\ga_\ia=\oga_\ia$ commutes with projective limits.

b) If $\ia$ is nil, then $\oga_\ia$ commutes with projective limits, but $\ga_\ia$ need not do so.

c) If $\ia$ is idempotent, then $\ga_\ia$ commutes with projective limits, but $\oga_\ia$ need not do so.

d) If $\ia$ is generated by idempotents, then $\ga_\ia=\oga_\ia$ commutes with projective limits.

e) If $R$ is absolutely flat, then $\ga_\ia=\oga_\ia$ commutes with projective limits.

f) If $R$ is noetherian, then $\ga_\ia=\oga_\ia$ need not commute with projective limits.
\end{prop}

\begin{proof}
a) follows from \ref{def80}.

b) The first claim follows from \ref{def80} c). For the second one, let $R$ and $\ia$ be as in \ref{nil25}. As $\ga_\ia(R)=\ia$ (\ref{cora158} C)), we have $(Y_i)_{i\in\N}\in\ia^\N=\ga_\ia(R)^\N$. If $n\in\N$, then $\prod_{i=n+1}^{2n}Y_i\in\ia^n$ and $\prod_{i=n+1}^{2n}Y_i\cdot Y_n\neq 0$, hence $\ia^n(Y_i)_{i\in\N}\neq 0$. Therefore, the canonical morphism $\ga_\ia(R^\N)\rightarrow\ga_\ia(R)^\N$ is not surjective. (Note that $\ia$ is quasinilpotent and T-nilpotent.)

c) If $\ia$ is idempotent, then $\ga_\ia(\bullet)\cong\hm{R}{R/\ia}{\bullet}$ (\ref{def30} C)), and the latter functor commutes with projective limits. For the second claim, let $\overline{R}$ and $\overline{\ia}$ be as in \ref{cora159} B) (cf. \ref{idem20} C)), so that $\overline{\ia}$ is idempotent, and let $N_i\dfgl\overline{R}/\langle Y_i\rangle_{\overline{R}}$ for $i\in\Z$. Then, $\oga_{\overline{\ia}}(N_i)=N_i$ for every $i\in\Z$, hence $\prod_{i\in\Z}\oga_{\overline{\ia}}(N_i)=\prod_{i\in\Z}N_i$. We consider $x=(1+\langle Y_i\rangle_{\overline{R}})_{i\in\Z}\in\prod_{i\in\Z}N_i$. If $j\in\Z$ and $k\in\N$, then for $i\in\Z$ with $j+k<i$ we have $Y_j^{2^k}=Y_{j+k}\notin\langle Y_i\rangle_{\overline{R}}$, hence $Y_j^{2^k}x\neq 0$. It follows that $x\notin\oga_{\overline{\ia}}(\prod_{i\in\Z}N_i)$ and thus the claim.

d) follows from c) and \ref{cora100} a).

e) follows from \ref{idem30} and d).

f) Let $R\dfgl\Z$, let $\ia\dfgl 2\Z$, and let $M_m\dfgl\Z/2^m\Z$ for $m\in\N$. Then, $\ga_\ia(M_m)=M_m$ for $m\in\N$, hence $\prod_{m\in\N}\ga_\ia(M_m)=\prod_{m\in\N}M_m$. For every $n\in\N$ there exists $m\in\N$ with $2^n\notin 2^m\Z$, hence $(1+2^m\Z)_{m\in\N}\in\prod_{m\in\N}\ga_\ia(M_m)\setminus\ga_\ia(\prod_{m\in\N}M_m)$, and thus $\ga_\ia=\oga_\ia$ does not commute with products (\ref{cora100} b)).
\end{proof}

\begin{prop}\label{exact110}
a) If $\ia$ is idempotent, then the following statements are equivalent: (i) $\ga_\ia$ commutes with inductive limits; (ii) $\ga_\ia$ commutes with right filtering inductive limits; (iii) $\ia$ is of finite type.

b) If $\ia$ is generated by idempotents, then the following statements are equivalent: (i) $\oga_\ia$ commutes with inductive limits; (ii) $\oga_\ia$ commutes with right filtering inductive limits; (iii) $\ia$ is of finite type.

c) If $R$ is absolutely flat, then the following statements are equivalent: (i) $\ga_\ia=\oga_\ia$ commutes with inductive limits; (ii) $\ga_\ia=\oga_\ia$ commutes with right filtering inductive limits; (iii) $\ia$ is of finite type.

d) If $\ia$ is idempotent and $\oga_\ia$ commutes with inductive limits, then $\ia$ need not be of finite type.
\end{prop}

\begin{proof}
a) As $\ia$ is idempotent, we have $\ga_\ia(\bullet)\cong\hm{R}{R/\ia}{\bullet}$. If $\hm{R}{R/\ia}{\bullet}$ commutes with right filtering inductive limits, then $R/\ia$ is of finite presentation (\cite[V.3.4]{sten}), thus $\ia$ is of finite type. If $\ia$ is of finite type, then it is generated by a single idempotent (\ref{idem20} B)), hence $R/\ia$ is projective (\cite[05KK]{stacks}), thus $\ga_\ia(\bullet)\cong\hm{R}{R/\ia}{\bullet}$ is exact and therefore commutes with inductive limits (\ref{exact30}). This yields the claim.

b) follows from \ref{cora100} a) and a).

c) follows from \ref{idem30} and b).

d) If $\ia$ is idempotent, nil and nonzero (e.g. as in \cite[2.2]{qr}), then $\ia$ is not of finite type, and the claim follows from \ref{def80} c).
\end{proof}

\begin{prop}\label{exact70}
a) If $\ia$ is nilpotent, then $\ga_\ia=\oga_\ia$ commutes with inductive limits.

b) If $\ia$ is nil, then $\oga_\ia$ commutes with inductive limits, but $\ga_\ia$ need not commute with right filtering inductive limits.

c) If $\ia$ has a power of finite type, then $\ga_\ia=\oga_\ia$ commutes with right filtering inductive limits.

d) If $R$ is noetherian, then $\ga_\ia=\oga_\ia$ commutes with right filtering inductive limits, but need not commute with inductive limits.
\end{prop}

\begin{proof}
a) follows from \ref{def80}.

b) The first claim follows from \ref{def80}. If $\ia$ is idempotent, nil and nonzero (e.g. as in \cite[2.2]{qr}), then $\ia$ is not of finite type, hence \ref{exact110} a) yields the second claim.

c) Let $I$ be a right filtering small category, and let $D\colon I\rightarrow\catmod(R)$ be a functor. By \cite[V.3.4]{sten} and our hypothesis, $\hm{R}{R/\ia^n}{\bullet}$ commutes with $I$-limits for infinitely many $n\in\N$. Keeping in mind that inductive limits commute with inductive limits, it thus follows \[\ga_\ia(\ilim_ID)\cong\ilim_{n\in\N}\hm{R}{R/\ia^n}{\ilim_ID}\cong\ilim_{n\in\N}\ilim_I\hm{R}{R/\ia^n}{D}\cong\]\[\ilim_I\ilim_{n\in\N}\hm{R}{R/\ia^n}{D}\cong\ilim_I(\ga_\ia\circ D).\] It is readily checked that the composition of these isomorphisms equals the canonical morphism $\ga_\ia(\ilim_ID)\rightarrow\ilim_I(\ga_\ia\circ D)$, and therefore the claim is proven. (A proof on foot of this statement is given in \cite[3.4.4]{bs}.)

d) The first claim follows from c). For the second one, we observe that applying $\ga_{2\Z}$ to the exact sequence of $\Z$-modules $0\rightarrow\Z\overset{2\cdot}\rightarrow\Z\rightarrow\Z/2\Z\rightarrow 0$ does not yield an exact sequence.
\end{proof}

\begin{prop}\label{exact107}
If\/ $\ga_\ia$ commutes with inductive limits, then it is a radical and $R=\ga_\ia(R)+\ia$.
\end{prop}

\begin{proof}
Suppose that $\ga_\ia$ commutes with inductive limits. If $M$ is an $R$-module, then applying $\ga_\ia$ to the exact sequence of $R$-modules \[0\rightarrow\ga_\ia(M)\rightarrow M\rightarrow M/\ga_\ia(M)\rightarrow 0\] yields an exact sequence of $R$-modules \[0\rightarrow\ga_\ia(M)\overset{\Id}\longrightarrow\ga_\ia(M)\rightarrow\ga_\ia(M/\ga_\ia(M))\rightarrow 0,\] hence $\ga_\ia(M/\ga_\ia(M))=0$. Thus, $\ga_\ia$ is a radical. Furthermore, applying $\ga_\ia$ to the exact sequence of $R$-modules \[0\rightarrow\ia\hookrightarrow R\rightarrow R/\ia\rightarrow 0\] and keeping in mind \ref{exact10} A) and \ref{exact20} yields an exact sequence of $R$-modules \[0\rightarrow\ia\cap\ga_\ia(R)\rightarrow\ga_\ia(R)\rightarrow R/\ia\rightarrow 0,\] hence $(\ga_\ia(R)+\ia)/\ia\cong\ga_\ia(R)/(\ia\cap\ga_\ia(R))\cong R/\ia$, and therefore $R=\ga_\ia(R)+\ia$.
\end{proof}

\begin{prop}\label{exact108}
a) Suppose that $\ia$ is prime. Then, $\ga_\ia$ commutes with inductive limits if and only if $\ia=0$.

b) Suppose that the $R$-module $R$ is of bounded small $\ia$-torsion. Then, $\ga_\ia$ commutes with inductive limits if and only if $\ia$ has a power that is generated by a single idempotent.

c) Suppose that $R$ is local, or that $R$ is semilocal and $\ia$ has a power of finite type. Then, $\ga_\ia$ commutes with inductive limits if and only if $\ia=R$ or $\ia$ is nilpotent.
\end{prop}

\begin{proof}
a) Suppose that $\ga_\ia$ commutes with inductive limits, so that $R=\ga_\ia(R)+\ia$ by \ref{exact107}. Applying $\bullet_\ia$ and keeping in mind \ref{flat15} we get $R_\ia=\ga_\ia(R)_\ia+\ia_\ia\subseteq\linebreak\ga_{\ia_\ia}(R_\ia)+\ia_\ia\subseteq R_\ia$, hence $R_\ia=\ga_{\ia_\ia}(R_\ia)+\ia_\ia$. As $\ia_\ia$ is the unique maximal ideal of $R_\ia$, we get $\ga_{\ia_\ia}(R_\ia)=R_\ia$. Therefore, $\ia_\ia$ is nilpotent, hence $\ia_\ia=0$, and thus $\ia=0$. The converse is clear by \ref{def80} b).

b) Suppose that $\ga_\ia$ commutes with inductive limits. There exists $n\in\N$ with $\ia^n\ga_\ia(R)=0$. Then, $\ia^n=\ia^n R=\ia^n(\ga_\ia(R)+\ia)=\ia^n\ga_\ia(R)+\ia^n\ia=\ia^{n+1}$ (\ref{exact107}), hence $\ia^n$ is idempotent. It follows from \ref{def50}, \ref{exact110} a) and \ref{idem20} B) that $\ia^n$ is generated by a single idempotent. The converse is clear by \ref{def50} and \ref{exact110} a).

c) Suppose that $\ga_\ia$ commutes with inductive limits and that $\ia\neq R$. First, we consider the case that $R$ is local, and we denote its maximal ideal by $\im$. Applying $\ga_\ia$ to the exact sequence of $R$-modules \[0\rightarrow\ia\hookrightarrow R\rightarrow R/\ia\rightarrow 0\] yields an exact sequence of $R$-modules \[0\rightarrow\ga_\ia(\ia)\rightarrow\ga_\ia(R)\rightarrow R/\ia\rightarrow 0.\] So, there exists $x\in\ga_\ia(R)$ with $x-1\in\ia\subseteq\im$, implying $x\notin\im$, hence $x\in R^*$, and therefore $\ga_\ia(R)=R$. It follows that $\ia$ is nilpotent. Next, we consider the case that $R$ is semilocal and that $\ia$ has a power of finite type. If $\im$ is a maximal ideal of $R$ with $\ia\subseteq\im$, then $\ia_\im\neq R_\im$ and $\ga_\ia(\bullet)_\im\cong\ga_{\ia_\im}(\bullet)$ (\ref{flat42} a)), hence $\ga_{\ia_\im}(\bullet)$ is exact, and thus what we have shown already implies that $\ia_\im$ is nilpotent. So, as $R$ is semilocal, there exists $m\in\N$ such that for every maximal ideal $\im$ of $R$ we have $(\ia_\im)^m=(\ia^m)_\im=0$, and it follows $\ia^m=0$. The converse is clear by \ref{def80}.
\end{proof}

\begin{no}\label{exact100}
A) If $\ga_\ia$ commutes with inductive limits, then $\ia$ need not be of finite type, as exemplified by a nilpotent ideal that is not of finite type (\ref{def80}).\smallskip

B) If $\oga_\ia$ commutes with inductive limits, then $\ia$ need not have a power of finite type, as exemplified by a nil ideal that is not nilpotent (\ref{def80}).
\end{no}

\begin{no}\label{exact130}
A) We know from \cite[3.7]{qr} that Grothendieck's Vanishing Theorem for small local cohomology, i.e., the right derived cohomological functor of the small torsion functor (cf. \cite[6.1.2]{bs}), need not hold over a non-noetherian ring. We can use \ref{exact110} to reach the same conclusion for small and large local cohomology. Indeed, let $R$ be absolutely flat and non-noetherian, and suppose that $\ia$ is not of finite type. Note that $\dim(R)=0$. By \ref{idem30} and \ref{exact110} c), the functor $\ga_\ia=\oga_\ia$ does not commute with inductive limits and hence is not exact. So, there exists an epimorphism of $R$-modules $h$ such that $\ga_\ia(h)$ is not an epimorphism, implying $H^1_\ia(\ke(h))\neq 0$ (where $H^1_\ia$ denotes the first right derived functor of $\ga_\ia$). Thus, Grothendieck's Vanishing Theorem does not hold over $R$.\smallskip

B) Hartshorne's Vanishing Theorem for small local cohomology (cf. \cite[3.3.3]{bs}) need not hold over a non-noetherian ring, for it would imply that if $\ia$ is nil, then $\ga_\ia$ is exact, contradicting \ref{exact70} b). 
\end{no}

\begin{qus}\label{exact140}
A) The preceding results give rise to the following questions (whose answers will probably require a study of small and large local cohomology):
\begin{aufz}
\item[($*$)] \textit{For which $\ia$ does $\ga_\ia$ or $\oga_\ia$ commute with projective limits?}
\item[($**$)] \textit{For which $\ia$ does $\ga_\ia$ or $\oga_\ia$ commute with (right filtering) inductive limits?}
\end{aufz}

B) The above results \ref{exact108} and \ref{exact100} give rise to the following more specific question:
\begin{aufz}
\item[$(*)$] \textit{Suppose that $\ga_\ia$ commutes with inductive limits. Does $\ia$ need to have a power of finite type?}
\end{aufz}
I conjecture its answer to be negative but was not able to find a counterexample.
\end{qus}


\noindent\textbf{Acknowledgements:} For their various help, mostly via {\tt MathOverflow.net}, I thank Andreas Blass, Yves Cornulier, Pace Nielsen, Pham Hung Quy and Will Sawin.



\begin{thebibliography}{99}

\bibitem{lipman} L. Alonso Tarr\'io, A. Jerem\'ias L\'opez, J. Lipman, \textit{Local homology and cohomology on schemes.} Ann. Sci. \'Ec. Norm. Sup\'er. (4) 30 (1997), 1--39.

\bibitem{a} N. Bourbaki, \textit{\'El\'ements de math\'ematique. Alg\`ebre. Chapitres 1 \`a 3.} Masson, Paris, 1970. \textit{Chapitre 8.} Springer, Berlin, 2012. \textit{Chapitre 10.} Masson, Paris, 1980.

\bibitem{ac} N. Bourbaki, \textit{\'El\'ements de math\'ematique. Alg\`ebre commutative. Chapitres 1 \`a 4.} Masson, Paris, 1985. \textit{Chapitres 5 \`a 7.} Herman, Paris, 1975. \textit{Chapitres 8 et 9.} Masson, Paris, 1983.

\bibitem{bs} M. P. Brodmann, R. Y. Sharp, \textit{Local cohomology (second edition).} Cambridge Stud. Adv. Math. 136. Cambridge Univ. Press, Cambridge, 2013.

\bibitem{ghr} R. Gilmer, W. Heinzer, M. Roitman, \textit{Finite generation of powers of ideals.} Proc. Amer. Math. Soc. 127 (1999), 3141--3151.

\bibitem{goodearl} K. R. Goodearl, \textit{Distributing tensor product over direct product.} Pacific J. Math. 43 (1972), 107--110.

\bibitem{hutchins} H. C. Hutchins, \textit{Examples of commutative rings.} Polygonal Publ. House, Passaic, 1981.

\bibitem{lazard} D. Lazard, \textit{Autour de la platitude.} Bull. Soc. Math. France 97 (1969), 81--128.

\bibitem{matsumura} H. Matsumura, \textit{Commutative ring theory.} Translated from the Japanese. Cambridge Stud. Adv. Math. 8. Cambridge Univ. Press, Cambridge, 1986.

\bibitem{qr} P. H. Quy, F. Rohrer, \textit{Injective modules and torsion functors.} Comm. Algebra 45 (2017), 285--298.

\bibitem{tor1} F. Rohrer, \textit{Torsion functors with monomial support.} Acta Math. Vietnam. 38 (2013), 293--301.

\bibitem{asstor} F. Rohrer, \textit{Assassins and torsion functors.} Acta Math. Vietnam. 43 (2018), 125--136.

\bibitem{part2} F. Rohrer, \textit{Assassins and torsion functors II.} In preparation.

\bibitem{ss} P. Schenzel, A.-M. Simon, \textit{Examples of injective modules for specific rings.} In preparation.

\bibitem{stacks} The Stacks Project Authors, \textit{Stacks project.} {\tt https://stacks.math.columbia.edu}

\bibitem{sten} B. Stenström, \textit{Rings of quotients.} Grundlehren Math. Wiss. 217. Springer, Berlin, 1975.

\bibitem{webb} C. Webb, \textit{Tensor and direct products.} Pacific J. Math. 49 (1973), 579--594.

\end{thebibliography}
\end{document}